\documentclass{amsproc}
\usepackage{eurosym}
\usepackage{amssymb}
\usepackage{amsfonts}
\usepackage{cmtt}

\theoremstyle{plain}

\newtheorem{theorem}{Theorem}[section]

\newtheorem{definition}[theorem]{Definition}

\newtheorem{lemma}[theorem]{Lemma}

\newtheorem{proposition}[theorem]{Proposition}
\newtheorem{remark}[theorem]{Remark}

\numberwithin{equation}{section}

\begin{document}

\title[Betti numbers of an Arf semigroups]{On the relation between Betti numbers of an Arf semigroup and its blowup}

\author[Micha\l\ Laso\'{n}]{Micha\l\ Laso\'{n}}
\address{Institute of Mathematics of the Polish Academy of Sciences, \'{S}niadeckich 8, 00-956 Warszawa, Poland}
\address{Theoretical Computer Science Department, Faculty of Mathematics and Computer Science, Jagiellonian University, S. {\L}ojasiewicza 6, 30-348 Krak\'{o}w, Poland}
\email{michalason@gmail.com}



\thanks{Author is supported by the grant of Polish MNiSzW N N201 413139.} 

\date{2011/08/10}

\def\n{\mathbb{N}}
\def\k{\mathbb{K}}

\begin{abstract}
In this note we prove the relation between Betti numbers of an Arf semigroup $S$ and its blowup $S'$ in the case
when they have the same multiplicity $n$. The relation is then $\beta_{i,s}(S')=\beta_{i,{s+(i+1)n}}(S)$.
\end{abstract}

\maketitle

\section{Introduction}

\begin{definition} A numerical semigroup is a subset of $\n$ that is closed under addition, contains zero and with finite
complement in $\n$.\end{definition}

In fact it is even a monoid, but in this setting it is usually called as above.
A subset $T$ of $S$ generates it if every element of $S$ is a linear combination of elements of $T$ with integer
coefficients. Every generating set $T$ contains the minimum set of generators, denoted by $G(S)$, which are exactly nonzero
elements that can not be represented as a sum of two nonzero elements of $S$. 

\begin{definition} The smallest generator in $G(S)$ is called the multiplicity of $S$ and the number of elements in $G(S)$ the
embedding dimension of $S$.\end{definition}

\begin{remark} The multiplicity is greater or equal to embedding dimension. \end{remark}

\begin{definition} An Arf semigroup is a numerical semigroup $S$ such that for each $n\in S$ the set $S(n)$ is also a semigroup, where $S(n)=\{s-n:\;s\in S,\;s\geq n\}$.\end{definition}

Arf property is equivalent to: if $s,t,u \in S$ and $s,t\geq u$ then $s+t-u\in S$.

\begin{proposition} For Arf semigroups multiplicity and embedding dimension are equal.\end{proposition} \begin{proof} Let $n_1$ be the multiplicity
of the Arf semigroup $S$, we will show that embedding dimension is greater or equal to $n_1$, which is sufficient due
to the above Remark. Let $r$ be any residue class modulo $n_1$ and let $S_r=\{s\in S:s\equiv r$ mod $n_1\}$. The
set $S_r$ is non empty so there exists a minimal element $s$ in it. If $s$ was not a generator then $s=t+u$ for
$t,u\in S\setminus\{0\}$, so thanks to the Arf condition $s-n_1=t+u-n_1\in S_r$, which is a contradiction. Hence for
every residue class modulo $n_1$ there is a generator belonging to it. \end{proof}

\begin{definition} Let $S$ be a semigroup with $G(S)=\{n_1,\dots,n_k\}$. Here and in the sequel we assume that
$n_1<\dots<n_k$. Then the blowup of $S$ is a semigroup $S'$ generated
by $n_1,n_2-n_1,\dots,n_k-n_1$. \end{definition}

\begin{proposition}\label{blowup} The blowup of an Arf semigroup $S$ with multiplicity $n_1$ is a semigroup $S''=S(n_1)=\{s-n_1:\;s\in S,\;s\geq
n_1\}$. \end{proposition} \begin{proof} Let $G(S)=\{n_1,\dots,n_k\}$, $S'$ the blow up of $S$ is contained in $S''$ because all its
generators are and $S''$ is a semigroup since $S$ is Arf. Elements of $S''$ are clearly contained in $S'$.\end{proof}

\begin{proposition} The blowup of an Arf semigroup is an Arf semigroup. \end{proposition} \begin{proof} Let $s',t'\geq u'$ and $s',t',u'\in S'$ the blowup
of an Arf semigroup $S$ with multiplicity $n_1$. Then $s'=s-n_1, t'=t-n_1, u'=u-n_1$ and clearly $s,t\geq u$.
Hence $s'+t'-u'=(s+t-u)-n_1\in S'$. \end{proof}

\begin{remark}\label{same mult} If $G(S)=\{n_1,\dots,n_k\}$ for an Arf semigroup $S$ which has
the same multiplicity as its blowup $S'$ then $G(S')=\{n_1,n_2-n_1,\dots,n_k-n_1\}$. It is because
both have the same multiplicity and because they are Arf they have also the same number of generators. In
general multiplicity of a blowup can decrease, then some of $n_i-n_1$ are not generators any more. This remark and Propositions \ref{blowup}, \ref{pomoc} are the
only places where the assumption that both semigroups have the same multiplicity is used.\end{remark}

\begin{proposition}\label{pomoc}
For an Arf semigroup $S$ with $G(S)=\{n_1,\dots,n_k\}$ and its blowup $S'$ of the same multiplicity we have that 
\begin{itemize}
 \item $s\in S \Leftrightarrow s-n_1\in S'$ unless $s=0$
 \item $s\in S \Leftrightarrow s-2n_1\in S'$ unless $s=0$ or $s=n_l$
 \item $s\in S \Leftrightarrow s\in S'$ unless $s=n_l-n_1$ for $l\neq 1$.
\end{itemize}
\end{proposition} 
\begin{proof}
The first equivalence follows from Proposition \ref{blowup}, to prove the second one it is enough to see that if  $s=w+w'$ for $w,w'\in S\setminus\{0\}$ then
$s-2n_1=(w-n_1)+(w'-n_1)\in S'$. And also $n_1,n_l-n_1=n_1+(n_l-2n_1)$ are generators of $S'$ so $n_l-2n_1$ does not belong to $S'$. The third one follows from the fact that if 
$s=w+w'$ for $w,w'\in S'\setminus\{0\}$ then due to Arf property of $S'$ $w+w'-n_1\in S'$ and so $s=(w+w'-n_1)+n_1\in S$, otherwise $s=n_j-n_1\notin S$.
\end{proof}

With a numerical semigroup $S$ with $G(S)=\{n_1,\dots,n_k\}$ and a field $k$ we associate the semigroup ring
$R=k[t^{n_1},\dots,t^{n_k}]$. Let $T$ be the graded polynomial ring $k[X_1,\dots,X_k]$ with $deg(X_i)=n_i$. Graded ring $R$ is isomorphic to $T/I$ where $I$ is the ideal describing relations between generators of $S$ (namely $I$ is generated by binomials $X^\alpha-X^\beta$ for which $\Sigma
n_i\alpha_i=\Sigma n_i\beta_i$). Then $Tor_i^T(R,k)$ gets a grading induced by the grading of $T$. We denote the part of degree $s$ by $Tor_i^T(R,k)_s$.

\begin{definition} The Betti numbers are $\beta_{i,s}=dim_kTor_i^T(R,k)_s$.\end{definition}

For $s\in S$ let $\Delta_s$ be the simplicial complex on the set of vertices $\{1,\dots,k\}$ consisting of faces
$\{n_{i_1},\dots,n_{i_j}\}$ such that $s-(n_{i_1}+\cdots+n_{i_j})\in S$.

\begin{lemma} \label{bh} $\beta_{i,s}=dim_k\tilde{H}_{i-1}(\Delta_s)$ \hfill$\square$\end{lemma}

In the above lemma $\tilde{H}$ denotes the reduced homology and $n_1,\dots,n_k$ has to be elements of $G(S)$, not only a
generating set. The above lemma is proven in \cite{mist} page 175 Theorem 9.2, see also \cite{brhe}. 

\section{Main result}

\begin{theorem} Let $S$ be an Arf semigroup, such that its blowup $S'$ has the same multiplicity $n_1$, then 
$\beta_{i,s}(S')=\beta_{i,{s+(i+1)n_1}}(S)$. \end{theorem} \begin{proof} Let $G(S)=\{n_1,\dots,n_k\}$, then by Remark \ref{same mult} 
$G(S')=\{n_1'=n_1,n_2'=n_2-n_1,\dots,n_k'=n_k-n_1\}$. Due to Lemma \ref{bh} we have to show that
$H_{i-1}(\Delta_s(S'))=H_{i-1}(\Delta_{s+(i+1)n_1}(S))$.

Let us fix $i$, denote $t=s+(i+1)n_1$ and compare the simplicial complexes $\Delta_s':=\Delta_s(S')$ and
$\Delta_t:=\Delta_t(S)$. Since we are looking on the $(i-1)$ homologies we can remove from both of them all
simplexes of dimension greater then $i$ and look only on faces of dimension $(i-1)$ - possible cycles or of dimension $i$ -
which give possible boundaries. For $m=i,i+1$ let us define the partial matching
$$\Delta_s'\ni \{j_1,\dots,j_m\} \leftrightarrow \{j_1,\dots,j_m\}\in \Delta_t$$ and let us
classify unmatched simplexes.
\newline

For $m=i$ there are two cases.\newline
First case $j_1\neq 1$. By Proposition \ref{pomoc} we have 
$$\{j_1,\dots,j_i\}\in\Delta_t \Leftrightarrow t-n_{j_1}-\cdots-n_{j_i}=u\in S \Leftrightarrow$$
$$\Leftrightarrow s-(n_{j_1}-n_1)-\cdots-(n_{j_i}-n_1)=u-n_1\in S' \Leftrightarrow \{j_1,\dots,j_i\}\in\Delta_s'$$ unless $u=0$.\newline
Second case $j_1=1$. Again by Proposition \ref{pomoc} we have that $$t-n_1-\cdots-n_{j_i}=u\in S \Leftrightarrow s-n_1-\cdots-(n_{j_i}-n_1)=u-2n_1\in S'$$
unless $u=0$ or $u=n_l$. 
\newline

For $m=i+1$ there are also two cases.\newline
First case $j_1\neq 1$. $$t-n_{j_1}-\cdots-n_{j_i}=u\in S \Leftrightarrow s-(n_{j_1}-n_1)-\cdots-(n_{j_i}-n_1)=u\in S'$$ unless $u=n_l-n_1$.\newline
Second case $j_1=1$. $$t-n_1-\cdots-n_{j_i}=u\in S \Leftrightarrow s-n_1-\cdots-(n_{j_i}-n_1)=u-n_1\in S'$$
unless $u=0$.
\newline

To conclude there are four types of unmatched faces:
\begin{enumerate}
 \item $\Delta_t$ has extra $(i-1)$-dimensional face $\{j_1,\dots,j_i\}$ for $t-n_{j_1}-\cdots-n_{j_i}=0$
 \item $\Delta_t$ has extra $(i-1)$-dimensional face $\{1,\dots,j_i\}$ for $t-n_1-\cdots-n_{j_i}=n_l$
 \item $\Delta_t$ has extra $i$-dimensional face $\{1,\dots,j_{i+1}\}$ for $t-n_1-\cdots-n_{j_{i+1}}=0$
 \item $\Delta_s'$ has extra $i$-dimensional face $\{j_1,\dots,j_{i+1}\}$ for $t-n_1-\cdots-n_{j_{i+1}}=n_l-n_1$ and
$j_r\neq 1$
\end{enumerate}

Let us observe that removing extra faces from the simplicial complexes $\Delta_t$ and $\Delta_s'$ does not change
$(i-1)$ homology.

The face $F_0$ of type 1 $\{j_1,\dots,j_i\}$ for $t-n_{j_1}-\cdots-n_{j_i}=0$ is not a member of any cycle. Namely if $\delta(\Sigma
\alpha_rF_r+\alpha_0F_0)=0$ we look at the coefficient of $\{j_2,\dots,j_i\}$. If it was a boundary of a
face $\{l,j_2,\dots,j_i\}\in \Delta_t$ we would get $t-n_l-n_{j_2}-\cdots-n_{j_i}=n_l-n_{j_1}\in S$, which is a
contradiction. So this coefficient is $\alpha_0$ and it is equal to $0$ hence removing $F_0$ does not affect
$(i-1)$ homology.

The face $F_0$ of type 4 $\{j_1,\dots,j_{i+1}\}$ for $t-n_{j_1}-\cdots-n_{j_{i+1}}=n_l-n_1$ for $j_r\neq 1$ does
not create any new boundary in $\Delta_s'$. We have that for any $r=1,\dots,i+1$
$$t-n_1-n_{j_1}-\cdots-n_{j_{r-1}}-n_{j_{r+1}}-\cdots-n_{j_{i+1}}=(n_l-n_1)+(n_r-n_1)\in S'$$ so $\{1,j_1,\dots,j_{r-1},j_{r+1},\dots,j_{i+1}\}\in
\Delta_s'$ and is matched. \newline We have that $\delta(\{1,j_1,\dots,j_{i+1}\})=\Sigma \alpha_jF_j+\alpha_0F_0$ is a
linear combination of matched faces and $F_0$ and the boundary of this linear combination is equal to $0$, so
$\Sigma \alpha_r\delta F_r+\alpha_0\delta F_0=0$, hence $\delta F_0$ does not enlarge dimension of subspace of
boundaries.

Consider the face $F_0$ of type 2 $\{1,\dots,j_i\}$ for $t-n_1-\cdots-n_{j_i}=n_l$. As in the case of faces of type 1 we
look for other faces with $\{j_2,\dots,j_i\}$ as a member of its boundary. If there is no such then $F_0$ is not a
member of any cycle and removing it does not change $(i-1)$ homology. If there is one, with an extra vertex $r\neq 1$, then $t-n_r-\cdots-n_{j_i}=n_l+n_1-n_r=u\in S$. We
get $n_l-n_1=(u-n_1)+(n_r-n_1)\in S'+S'$ so $u=n_1,r=l$, since $n_l-n_1$ is a generator of $S'$. Hence $l\neq
j_2,\dots,j_i$ and all such faces $F_0$ are members of a boundary of a face of type 3 $\{1,\dots,j_i,l\}$ for
$t-n_1-\cdots-n_{j_i}-n_l=0$. We have also that two different faces of type 3 have disjoint boundaries, because
they have sum equal to $0$. So faces of type 2 and 3 can be considered together separately for different faces
of type 3. \newline
Let us fix a face of type 3 $\{1,j_2,\dots,j_{i+1}\}$ and $i$ its subfaces of type 2. We want to show that adding them does not change $(i-1)$ homology.
By looking at the boundary simplexes $\{j_2,\dots,j_{r-1},j_{r+1}, \dots,j_{i+1}\}$ for $r=2,\dots,i+1$ we get $i$ equations that has to be satisfied for a linear combination to be a cycle. So by adding extra faces to $\Delta_t$ a cycle can contain $i+1$ new simplexes of dimension $(i-1)$ but has to satisfy $i$ new linearly independent equations. This enlarge the dimension of kernel of $\delta_{i-1}$ by at most one. But $\delta_i(\{1,j_2,\dots,j_{i+1}\})$ is a boundary so it is also a
cycle, hence the dimension of kernel is greater by exactly one. The dimension of the image of $\delta_i$ is also greater by
one since by adding a face of type 3 we add a new boundary which contains new faces, hence the $(i-1)$ homology is the same.\end{proof}

\section*{Acknowledgements}

We would like to thank very much Ralf Fr\"{o}berg for introduction to this subject, posing this problem and many
interesting talks. We are also grateful to Mats Boij and Alexander Engstr\"{o}m for their help.


\bigskip
\begin{thebibliography}{Wi2'}
\bibitem{brhe} W.~Bruns, J.~Herzog \emph{Semigroup rings and simplicial complexes}, J. Pure and App. Algebra 122 (1997), 185-208.
\bibitem{mist} E.~Miller, B.~Sturmfels \emph{Combinatorial Commutative Algebra}, 2004.
\end{thebibliography}
\end{document}